\documentclass[letterpaper]{article}   

\usepackage{wrapfig}
\usepackage{graphics} 
\usepackage{epsfig} 
\usepackage{mathptmx} 
\usepackage{times} 
\usepackage{amsmath} 
\usepackage{amssymb}  
\usepackage{subfigure}
\usepackage{amsthm}
\usepackage{authblk}

\newtheorem{theorem}{Theorem}

\newtheorem{assumption}[theorem]{Assumption}
\usepackage{color}

\newtheorem{corollary}[theorem]{Corollary}

\newtheorem{definition}[theorem]{Definition}

\newtheorem{lemma}[theorem]{Lemma}

{ 
\newtheorem{remark}[theorem]{Remark}
\newtheorem{example}[theorem]{Example}
}
\begin{document}
%
\title{ Contingency Analysis of Power Networks : A System Theoretic Approach}
%
%
%

\author[1]{Sambarta Dasgupta\thanks{dasgupta.sambarta@gmail.com}}
\author[2]{Umesh Vaidya\thanks{ugvaidya@iastate.edu}}
\affil[1]{Monsanto Company}
\affil[2]{ECPE Department, Iowa State University}

\maketitle
\begin{abstract}
In this work, we have proposed a system theoretic method  to compute sensitivities of different lines for $N-k$ contingency analysis in power network. We have formulated the $N-k$ contingency analysis as the stability problem of power network with uncertain links. We have derived a necessary condition for stochastic stability of the power network  with the link uncertainty. The necessary condition is then used to rank order the contingencies. We have shown due to interaction between different uncertainties the ranking can substantially change. The state of the art  $N-k$ contingency analysis does not consider the possibility of interference between link uncertainties and rank the links according to the severity of $N-1$ contingencies. We have presented simulation results for New England $39$ bus system as a support of our claim. 
\end{abstract}
\section{introduction} \label{sec_intro}
Multiple line outages in power grid can potentially result in cascade failure, and blackouts \cite{cascade_1}.  As a consequence, contingency analysis of the power network is central to the planning and operation of power systems \cite{wood2012power}. Also, identification of  critical threats to power gird under cyber attacks is necessary for safe operation  \cite{pasqualetti2011cyber}. The main premise of contingency analysis, in the context of power system, is studying the effect on the stability and performance of the power network under line outage and faults. The analysis with single link uncertainty is termed as $N-1$ contingency, whereas with multiple link uncertainties is called $N-k$ contingency ($k$ line outage among the total $N$ lines in the network).  There are various different heuristic methods to identify critical links in contingency analysis \cite{chen2005identifying, wu2011probability}. Most of the existing techniques make a tacit assumption that the relative ordering of critical links are preserved as the domain of interest is changed from $N-1$ to $N-k$ contingency analysis. Roughly speaking, this assumption is equivalent to a superposition of uncertainties in a network. In reality, this seemingly plausible assumption might be misleading.  In this work, we have demonstrated, the interaction of uncertainties in the network plays a vital role in determining the criticality. We have provided a systematic system theoretic way of computing sensitivities for different links in power network based on Lyapunov equation for $N-k$ contingency analysis. \\ 
A power network comprises of buses and transmission lines, which can be perceived as nodes and edges respectively. The problem of $N-k$ contingency analysis could be formulated as a stability problem of a network with multiple edges uncertain. This problem could further be transformed to a stability problem of a discrete time linear time varying (LTI) system with multiple sources of uncertainties. The corresponding stochastic stability problem, involving linear time invariant dynamics,  requires simultaneous search for Lyapunov function, and bounds on uncertainty to guarantee mean square stability of stochastic Linear system \cite{Boyd_book}.  \\ In this paper, a necessary condition for mean square stability is expressed in terms of the solution to  Lyapunov equation and relative measure of variance of uncertainties. The key feature of the stability condition is that the Lyapunov equation is decoupled from the uncertainties. The necessary condition is used to rank order the uncertainties. The relative sensitivity of an uncertain parameter can be found by computing gramian based gains for various uncertainty sources.  Furthermore, we have proposed a technique to compute the interaction among different uncertainties. The proposed method is tested on contingency analysis for the transient stability over the New England $39$ bus system, which is a reduced power gird model of New England and part of Canada.
\section{Problem Formulation}\label{section_set-up}
The power network is comprised of generator and load buses, which are connected to one another by transmission lines. The buses can be thought as nodes and lines as the edges in a graph. The nodes in this set up can be represented as dynamical systems.  The dynamical systems, which are represented by nodes, are then coupled through the transmission lines, which are the edges in the network.  In $N-k$ contingency studies, we investigate such a system with $k$ number of uncertain links.  The $k^{th}$ node in the network is denoted as $S_k$, and can be described as ,
\begin{equation}
S_k=\left\{\begin{array}{ccl}
x_k (t+1)&=&A_k  x_k (t) +B_k u_k (t), \\
y_k (t) &=&C_k x_k(t),\;\;\;k=1, 2, \ldots,M. \label{component_dynamics}
\end{array}\right.
\end{equation}
where $x_k \in  \mathbb{R}^{n}$, $u_k\in  \mathbb{R}$, and $y_k  \in  \mathbb{R}$ are the state, input, and output of the $k^{th}$ component sub-system respectively. $A_k  \in  \mathbb{R}^{n \times n}$ is the system matrix of the $k^{th}$ subsystem. $B_k$ and $C_k$ are column and row vectors, such that $B_k, C_k \in \mathbb{R}^{n}$. The subsystems, which form the nodes of the network, form the inter connected network, as can be observed in Fig. \ref{sche1}. The input to $k^{th}$ sub-system is a linear combination of the outputs from the all the subsystems. The input to the $k^{th}$ sub-system, can be modeled as, $u_k (t)=\sum_{\ell=1, \ell \neq k}^{M} \mu_{k \ell}  \; \left ( a_{k \ell} y_k (t) + b_{k \ell} y_\ell (t) \right ), $ where, $a_{k \ell}, b_{k \ell}, \mu_{k \ell}$'s are scalar quantities (depending on the network topology these will take values). The input to a particular subsystem is comprised of the feedbacks coming from all the other subsystems. The contribution of link $ (k, \ell) $ to the $u_k$ is $ \mu_{k \ell}  \; \left ( a_{k \ell} y_k (t) + b_{k \ell} y_\ell (t) \right )$, which is a weighted sum of the outputs of the subsystems at its two ends, namely $y_k(t)$ and $y_{\ell}(t)$.  Next, for purpose of contingency analysis, we consider some of the links in the network are uncertain. Let us denote the set of all uncertain link as \[ \mathcal{S} := \{(k, \ell) | (k, \ell) \text {is an uncertain link} , \ell \neq k \} . \] The uncertainty in a link $(k, \ell) \in \mathcal{S}$ is modeled as $\delta_{k \ell}(t)$, which is an i.i.d. sequence of random variables with mean $0$, and variance $\sigma_{k \ell}^2 > 0$. The uncertain link is represented as $\xi_{k \ell}(t) := \mu_{k \ell} + \delta_{k \ell}(t)$.
The input with the uncertain link would become, $ u_k(t)  = \sum_{\ell =1 , \ell \neq k}^{M} \mu_{k \ell}  \; \left ( a_{k \ell} y_k (t) + b_{k \ell} y_\ell (t) \right )  + \sum_{\ell | (k, \ell) \in \mathcal{S} }  \delta_{k \ell}(t)  \; \left ( a_{k \ell} y_k (t) + b_{k \ell} y_\ell (t) \right ) . $  We define,   $ x (t) := [ x^T_1 (t), x^T_2(t), \ldots, x^T_M (t)]^T \in \mathbb{R}^{M n} .  $ 
 \begin{figure}[h]
    \centering
    \includegraphics[width=0.45\textwidth]{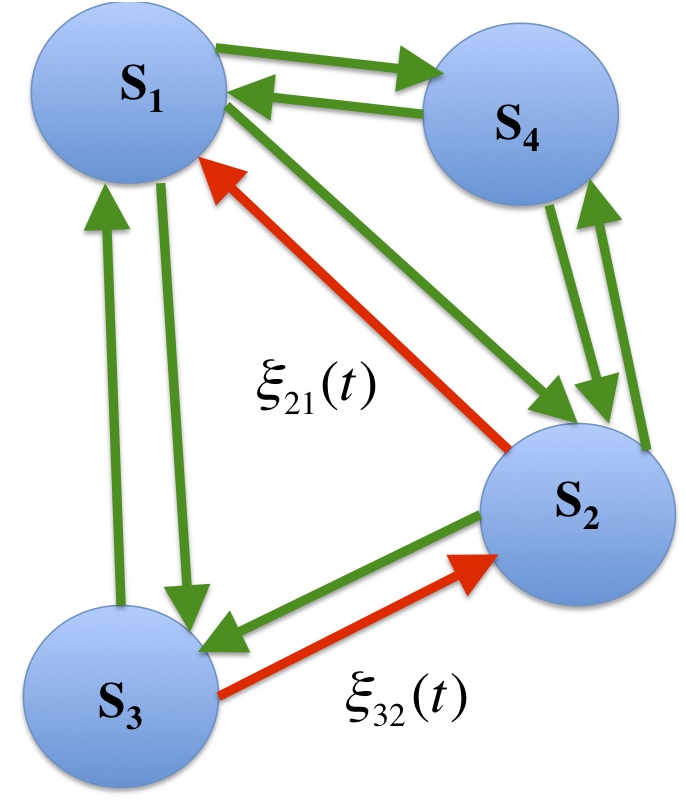}
  \caption{ Schematic of the Power network as an inter connected system, where the red edges represent different contingencies.}
 \label{sche1}
\end{figure}
 We also define,
\[ C_{k \ell} := (\underbrace{0 \ldots, 0}_{n (k-1)}, ~~ a_{k \ell} C_{ k}, 0, ~~\ldots, 0 ,  ~~ b_{k \ell} C_{\ell}, ~~ \underbrace{0 \dots, 0}_{(M- \ell) n }) \in \mathbb{R}^{Mn}. \]
With this setup input to $k^{th}$ subsystem could be written as,
\begin{align*}
 u_k(t) &= \sum_{\ell=1, \ell \neq k}^{M} \mu_{k \ell}   \;   C_{k \ell} \; x (t) + \sum_{\ell | (k ,\ell) \in \mathcal{S} }^{M} \delta_{k \ell}(t)   \;   C_{k \ell} \; x (t) .
 \end{align*}
Next, we combine the mean contribution of the links into the nominal dynamics and separate out the contribution of the uncertainty in the dynamics. The system Eqs. (\ref{component_dynamics})  can be written in compact form as following:
  \begin{equation}
   x(t+1) = A x(t) + \sum_{(k, \ell) \in \mathcal{S}} \delta_{ k \ell}(t)\;  \bar B_k \;  {C}_{k \ell} \; x (t),  \label{system_robust-form}
   \end{equation}
  where, $   A : = diag \left( A_1, \dots, A_M \right) + \sum_{k,\ell=1 , \ell \neq k }^{M} \; \mu_{k \ell} \; \bar B_k \;  C_{k \ell} \in \mathbb{R}^{Mn \times Mn},  $
and $\bar B_k$ is a column vector of size $M n$ and is obtained by stacking zero, $B_k$ column vector starting at $n(k-1)$ location.  The nominal network system, without presence of link uncertainty, shows stable behavior. Our goal is to find out relative amount of uncertainties, that can be tolerated without causing instability.  First we describe the stability assumption for the nominal deterministic system,
\begin{eqnarray}
x (t+1)=A  x(t). \label{nominal_system}
\end{eqnarray}
\begin{assumption}
The matrix $A$, as described in \eqref{nominal_system}, has all the eigenvalues inside the unit circle. 
\end{assumption}
We also make two more assumptions on the nominal system. The first one guarantees a uniform lower bound on system matrices. This technical assumption is needed in the proof of the Theorem \ref{theorem_main1}.
\begin{assumption} \label{ja_bound}
The matrix $A(t)$ from \eqref{nominal_system} is  lower bounded i.e.
\[ A^T A \ge L >  0, \forall t. \]
\end{assumption}
\begin{assumption} \label{assum_rec}
The pair $(A, C_{k \ell})$ is  observable \cite{paganini_dull}, for all $({k, \ell}) \in \mathcal{S}$.
\end{assumption}
The Assumption \ref{assum_rec} ensures the observability gramian exists and well defined for all input directions $C_{k \ell}$. The necessary condition, that we derive later, is in terms of the observability gramians. We make the Assumption \ref{assum_rec} in order to ensure the existence of those.
In order to analyze the stability of the system, described by \eqref{system_robust-form} we need the following notion of stochastic stability.
\begin{definition} \label{stab_def}
The system, described by \eqref{system_robust-form}, is mean square exponentially stable if there exists  $K>0$, and $\beta < 1$ such that,
\[ E \parallel x(t+1) \parallel^2 < K \beta^t \parallel x(0)\parallel^2 , \forall x(0). \]
\end{definition}
\section{Main Results} \label{main_res}
In this section, we state and prove necessary condition for mean square stability of \eqref{system_robust-form}. The necessary condition involves the gramians for various input directions and thebounds on the uncertainties.  
\begin{theorem}\label{theorem_main1}
The necessary condition for the mean square exponential stability of the  system (\ref{system_robust-form}) is given by,
\begin{align}
\sum_{ (k, \ell) \in \mathcal{S}}  \sigma_{k \ell}^2  \bar B^T_k {P}  \bar B_k  C_{k\ell}   C^T_{k\ell}  \le   \bar \alpha   \sum_{(k, \ell) \in \mathcal{S}}   C_{k\ell}   C^T_{k\ell},  \label{ellipse_condition}
\end{align}
 where,   the matrix $P = P^T > 0$, satisfies, $ A^T {P} A - P  =  -  \sum_{(i,j) \in \mathcal{S}} C^T_{i j}   C_{i j} , $
  and $\bar \alpha \ge 1$ is a scalar quantity.
\end{theorem}
\vspace{0.1 in}
The proof of  Theorem \ref{theorem_main1} is presented in the Appendix (Section \ref{appen}).
\begin{remark}
Theorem \ref{theorem_main1} provides a necessary condition for mean square exponential stability of the network, described by \eqref{system_robust-form}. Before outlining the proof of Theorem \ref{theorem_main1}, we discuss the insight and the implications of the theorem. The matrix ${P}$ is the observability gramian, for all the uncertainty injection directions. The term $\left (   \bar B^T_k P  \bar B_k  \right ) C_{i j}   C^T_{i j}  $ is the gain, seen by the $(k, \ell)$ link, if uncertainty is injected at all the links in $\mathcal{S}$, and this term accompanies $\sigma_{k \ell}^2$ in the expression. The implication of it is that the bound of the uncertainty for a particular link will be less, if a uncertainty injection at that link has large magnification. Also, it can be seen the \eqref{ellipse_condition} is coupled for all $\sigma_{k \ell}$'s, which allows to trade amount of uncertainties for different links. Increase in uncertainty in one link could be compensated for decrease in the other. This condition would be used to compute sensitivity of different links for $N-k$ contingency case. The following corollary gives the necessary condition for mean square stability for single link uncertain, which would be used to compute sensitivity for $N-1$ contingency.
\end{remark}
\begin{corollary} \label{SISO_con}
The necessary condition with single link $(k, \ell)$ uncertain is, 
\begin{align*} 
\sigma_{kl}^2 B_k^T  P^{k \ell} B _k  < 1,  ~ A^T {P}^{k \ell} A - P^{k \ell}   = -C_{k \ell}^T C_{k \ell} . 
 \end{align*}
\end{corollary}
Next, we would use the necessary condition for contingency analysis.
\section{$N-k$ ~ Contingency Analysis}
In this section we would apply the necessary condition to compute the sensitivities corresponding to various different links for $N-k$ contingency analysis. 
\subsection{Sensitivities of various links} \label{crit_idn}
 Equation (\ref{ellipse_condition}) provides us with a computable condition for computing relative sensitivity of uncertainty links. In particular, we notice that the Eq. (\ref{ellipse_condition}) is an equation of an ellipsoid and length of the axis along the direction of $\sigma_{k \ell}$ is proportional to,
\begin{eqnarray}
F_{k\ell}:=\left (   \bar B^T_k  {P}  \bar B_k    C_{k\ell}  C^T_{k\ell}    \right )^{-\frac{1}{2}}, ~~(k, \ell) \in \mathcal{S} . \label{rand_order}
\end{eqnarray}
which gives the relative sensitivity of link ${k\ell}$. The larger the value of $F_{k\ell}$ more uncertainty can be tolerated in the random parameter $\delta^{k\ell}$, which will make the link relatively less critical in the contingency analysis.
\subsection{Characterization of Uncertainty Interaction}
Next, we would propose an index to capture how much a set of uncertainties are interacting in the network. Let us consider a set of link uncertainties. If the uncertainties are present one at a time ($N-1$ contingency), then from Corollary \ref{SISO_con} the bounds on the uncertainties, $\sigma_{k \ell}$,  would be proportional to \[ S_{k \ell} := \left ( \bar B_k^T  P^{k \ell} \bar B _k \right)^{- \frac{1}{2}}. \] We define a vector $S$, which is composed of the $S_{k \ell}$'s. Again, if the uncertainties are working simultaneously ($N-k$ contigency), then $F_{k \ell}$'s provide the relative bound on $\sigma_{k \ell}$. We form another vector $F$ from $F_{k \ell}$'s. If the interaction among the uncertainties are small, then the vector $F$ and $S$ will be aligned, which would make the quantity $\frac{F^T S}{\parallel F \parallel \parallel S \parallel}$ very close to $1$. Or in other words,  $1- \frac{F^T S}{\parallel F \parallel \parallel S \parallel}$ would be very close to $0$. From this intuition, we define the index capturing the interaction as following,
\[ I := 1- \frac{F^T S}{\parallel F \parallel \parallel S \parallel} . \]
It can be observed, that $0\le I \le 1$, and higher value of the index signifies more interaction among the uncertainties.
\begin{example}
In the following example, we demonstrate the importance of identifying relative criticality of uncertainties, and interaction of uncertainties.
Let us consider the uncertain system, $x(t+1) = A x(t) + \delta_1 B_1 C_1 x(t) + \delta_2 B_2 C_2 x(t)$ , where, $\delta_1, \delta_2$ are i.i.d. random variables with $0$ mean, and variance $\sigma_1^2, \sigma_2^2$  respectively. Also, 
\begin{align}
& A=
\begin{bmatrix}
   -0.07   ~  1.00 ~  -0.23\\
   ~ 0.10~  ~  0.70 ~  -0.10\\
   -0.17  ~  1.00  ~  -0.13
\end{bmatrix},   \nonumber  \\
& B_1 = [- 1 ~ -1 ~ 0] ^T, ~~C_1 = [ 1 ~ 1 ~ 0 ], \nonumber \\
&  B_2 =[ -1 ~ 1 ~ 0 ], ~~ C_2 = [ -5 ~ -0.1 ~ 0.01 ]. \label{eq_example}
\end{align}
Eigenvalues of the matrix $A$ are $0.7,0.1, \text{and} - 0.1$ respectively, which makes the nominal system stable. The necessary and sufficient condition for mean square stability  for  given $\sigma_1^2, \sigma_2^2$ will be there exists a $\hat P= \hat P^T > 0$ such that \cite{Boyd_book}, 
\[ A^T \hat P A - \hat P + \sigma_1^2 B_1^T \hat P B_1 C_1 C_1 ^T + \sigma_2^2 B_2^T \hat P B_2 C_2 C_2 ^T < 0.  \]
From this condition, we compute the feasible values of the $\sigma_1, \sigma_2$, and plot them in Fig. \ref{fig_example} (a) . The maximum allowable $\sigma_1, \sigma_2$ can be computed from condition in Corollary \ref{SISO_con}, which involves computing the bounds taking one uncertainty at a time. From the plot, it can be observed that the parameter $\sigma_2$ is more sensitive than $\sigma_1$. Now, if for robust design one solves the problem of finding an upper bound on  $\sigma_1, \sigma_2 \le \sigma$, the value of $\sigma$ would be dominated by the relatively more critical uncertainty $\delta_2$. The red line passing through origin is the line $\sigma_2 = \sigma_1$ and the red square ( $\sigma_1, ~\sigma_2 \le \sigma = 0.11$ ) provides a bound on the uncertainties to guarantee stability. The size of the square is very much governed by the relatively sensitive uncertainty, which is $ {0.11}^2 = 0.01$.  Now, if we consider the relative sensitivity in the uncertainty and scale them accordingly, we can improve the area of the uncertainty region. We compute the sensitivity of the two parameters according to the formula \eqref{rand_order}. If we scale the uncertainties and try to find out modified bounds as, $\frac{\sigma_1}{F_1},\frac{ \sigma_2}{F_2} \le \sigma $, the resulting region turns out to be the one shown as the green rectangle. The construction is commenced by drawing the line $\sigma_2 = \frac{F_1}{F_2} \sigma_1$, and finding out the intersection point with the boundary of the feasibility region. The area of the green is $0.038$, which is more than the area for the red square. Also, we do similar construction with the line $ \sigma_2 = \frac{S_1}{S_2} \sigma_1$, where $S_1, S_2$ are computed considering one uncertainty at a time. The area of the blue rectangle turns out to be $0.032$. This shows by scaling the weights appropriately, for stability condition improved bounds could be obtained on the uncertainty.  The two extreme ends of the boundary curve of the set of feasible points $(\sigma_1, \sigma_2)$  could be found from the SISO condition (corollary \ref{SISO_con}). Next, we demonstrate the case when the interaction is less. We consider the second example, where the $A$ matrix is the same \eqref{eq_example},
\begin{align}
 B_1 & = [ 1 ~ 0 ~ -1] ^T, ~~ C_1 = [ -1 ~ 1 ~ 1 ], \nonumber \\
 B_2  & =[ -1 ~ -1 ~ 0 ], ~~ C_2 = [ -5 ~ -1 ~ 1 ]. \label{eq_example2}
\end{align}
 \begin{figure}[h]
\begin{center}
\mbox{
\subfigure[]{\scalebox{.8}{\includegraphics[width=2.5 in, height=2.4 in]{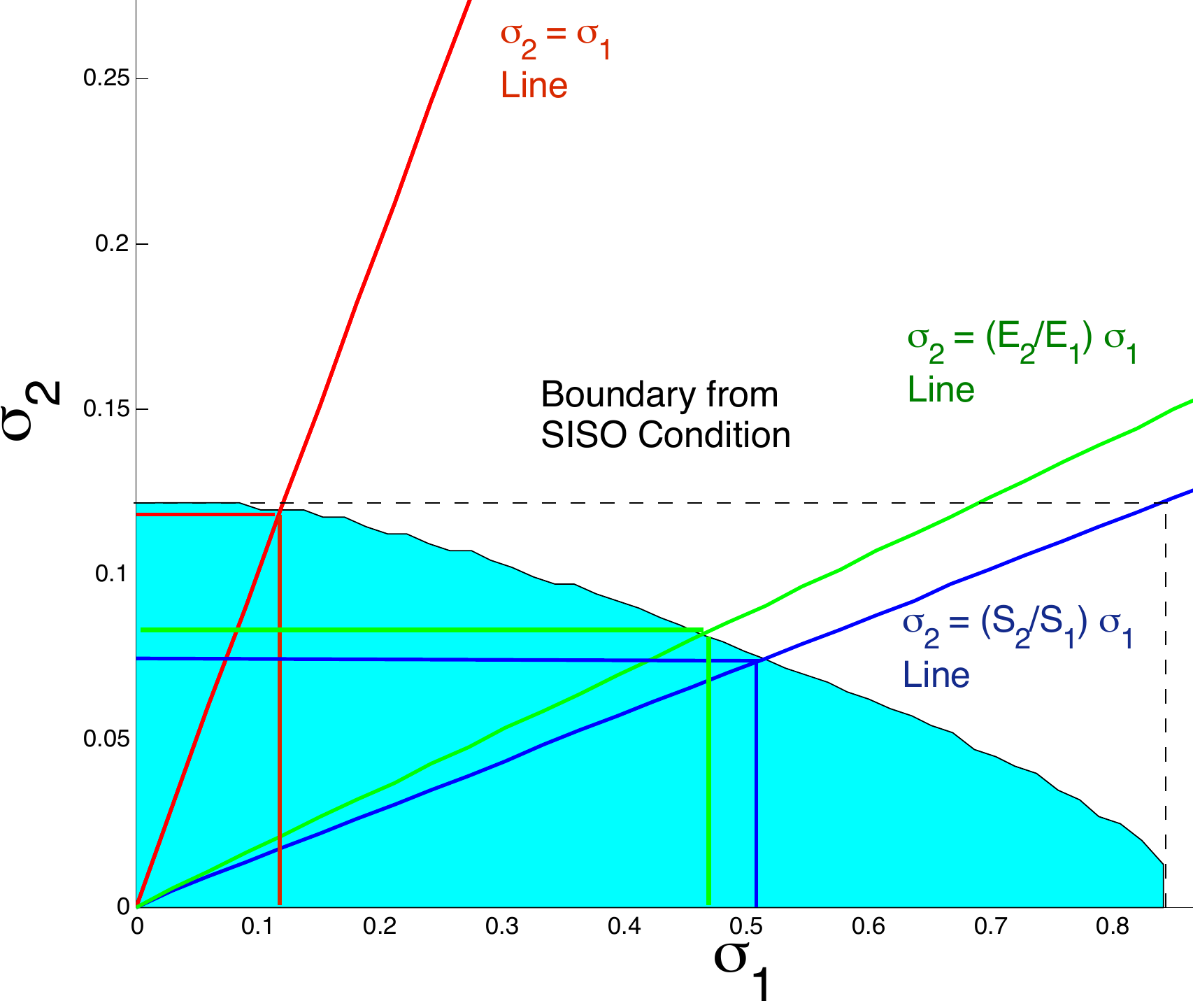}}}
\subfigure[]{\scalebox{.26}{\includegraphics[width=7.7 in, height=7.5 in]{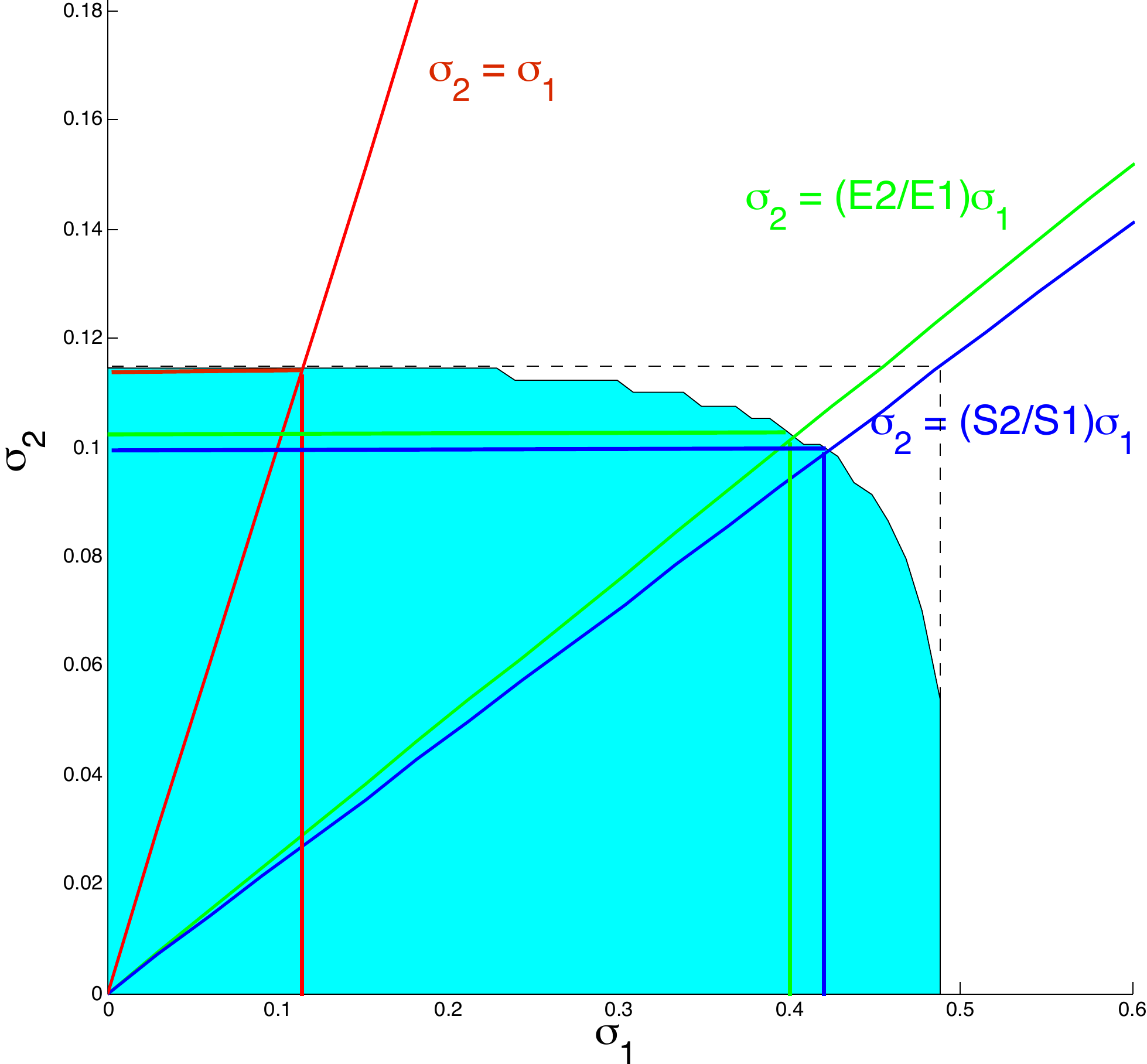}}} }
\caption{ The feasibility region of uncertainties when input output vectors - (a) described by  \eqref{eq_example},  (b) described by  \eqref{eq_example2} . }
\label{fig_example}
\end{center}
\end{figure}
In this case, the feasibility region along with the two rectangles are shown in Fig. \ref{fig_example} (b).  In this case the feasibility region almost fills the dotted rectangle, which comes from the SISO condition. This indicates the two uncertainties interact lesser. Also it can be noticed, the green and blue rectangles almost merge, as for lesser interaction the $\frac{F_2}{F_1} \approx \frac{S_2}{S_1}$. Comparing Fig. \ref{fig_example} (a)  and Fig. \ref{fig_example} (b), it can be concluded as the interaction becomes more the difference between SISO and MIMO bounds become more. In simulation section, we demonstrate that for multiple uncertainties, even the relative ranking can get completely changes due to interaction.
\end{example}
\section{Simulation Results} \label{sim_res}
The method, we have developed, is general in nature and could be used for contingency analysis for different types of stability problems in power network. In this section, we have considered the contingency analysis of transient stability of power network, which is has drawn research efforts in the past \cite{cont_transience_1}. For simulation purposes, we have chosen New England $39$ bus system.  We have considered  $N-k$ contingency problem for $k=4$, and have chosen two different set of contingencies. For one set of contingencies the raking for $N-1$, and $N-k$ remain the same, whereas for the second set of contingencies it completely changes. This suggests a strong need for considering uncertainty interaction into account, while ranking the uncertainties in $N-k$ scenario. 
\begin{figure}[h]
    \centering
    \includegraphics[width=0.50 \textwidth]{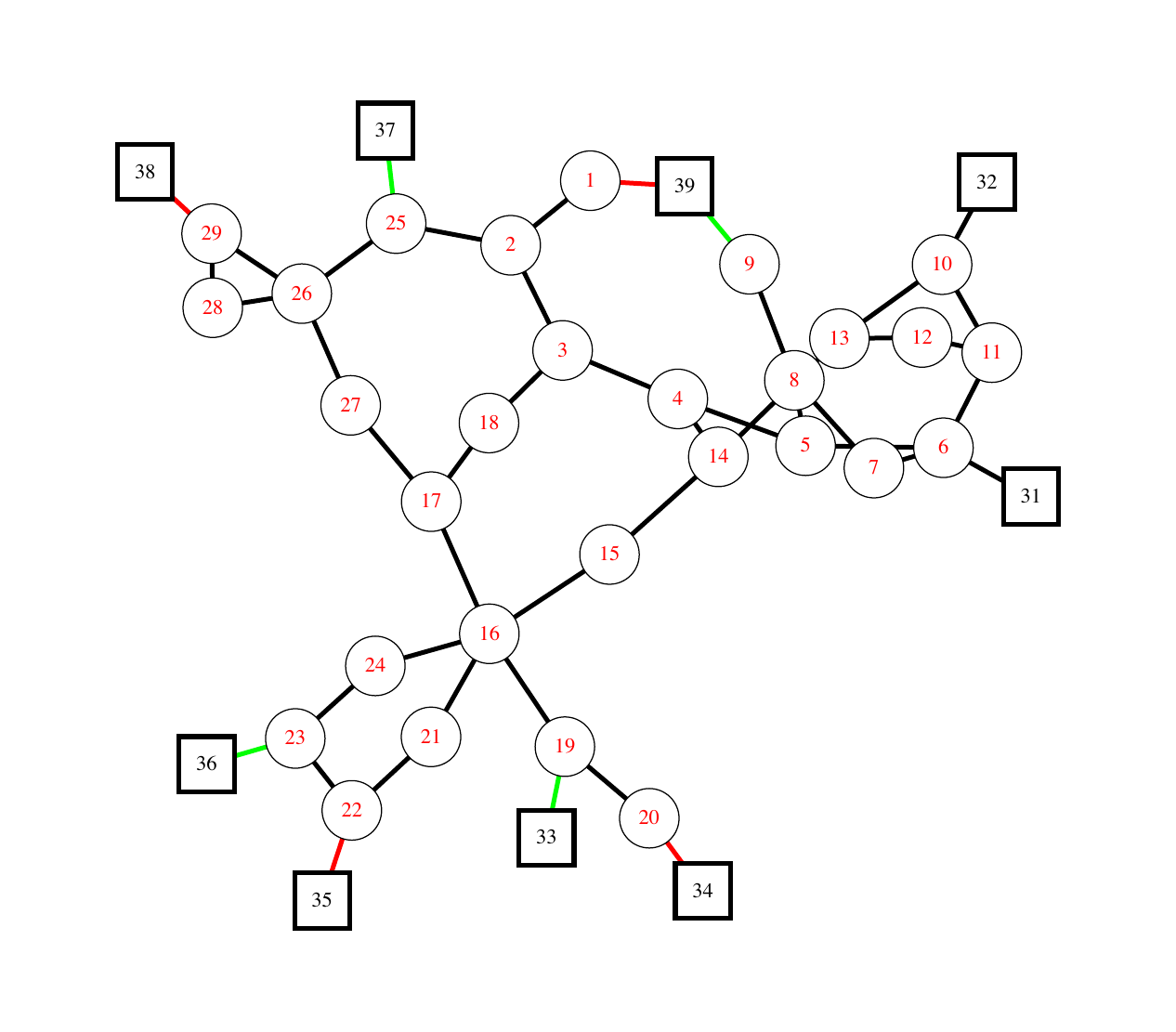}
  \caption{New England $39$ Bus system with load and generator buses are demarcated as squares and circles respectively, and also two set of contingencies are colored in green and red.}
 \label{fig_39_BUS_System}
\end{figure}
The dynamic model of the power network comprises of the rotor angle swing dynamics, as described in \cite{pasqualetti2011cyber}. In this case, the state vector consists  of $[\delta^T \omega^T \theta^T]^T$, where these are rotor angle, frequency and phase angles respectively. The network is comprised of 10 generator and 29 load buses, where some of the buses represent aggregate of individual generators and loads. The connection diagram of 39 bus system is shown in Fig. \ref{fig_39_BUS_System}. Let $\mathcal{L}$ be the laplacian matrix for the network. Also, the laplacian matrix could be decomposed as, $ \mathcal{L} := 
 \begin{bmatrix}
 \mathcal{L}_{ll} ~~\mathcal{L}_{lg}\\
 \mathcal{L}_{gl} ~~ \mathcal{L}_{gg}
 \end{bmatrix} . $
Under steady state conditions, we consider the linearized dynamics. The phase angles could be eliminated from the state equation using Kron reduction. After doing Kron reduction, as described in \cite{pasqualetti2011cyber},
\begin{equation*}
  \begin{bmatrix}
\dot\delta^T (t)\\
\dot \omega^T (t)\\
 \end{bmatrix} = -
  \begin{bmatrix}
 0 ~~~ ~~~~~~~~~~~~~~~~~~~~~~~-I \\
  M^{-1} \left ( \mathcal{L}_{gg} - \mathcal{L}_{gl} \mathcal{L}^{-1}_{ll}\mathcal{L}_{lg} \right)  ~~ M^{-1}D_g \\  
 \end{bmatrix} 
\begin{bmatrix}
\delta^T (t)\\
 \omega^T (t)\\
 \end{bmatrix}
\end{equation*}
The reduced state equation describes the linearized rotor angle and frequency evolution at generator buses. The reduced system has $10$ generators as nodes and each of the generators has $2$ states, namely rotor angle and frequency.  It can be noted that the inversion of $L_{ll}$ matrix, which was sparse in nature, results in a fully connected reduced system.
The states of the reduced system is denoted as $x =[\delta^T \omega^T ]^T$, and the state equation is of the form $\dot x(t) = A x(t)$. Now we discretize the state equation as $x(t+1) = (I+A \Delta t ) x(t)$, where $\Delta t $ is the sampling period. We use this discretized model for purpose of computation. \\
For simulation purposes, we consider New England 39 bus system, which is a reduced model for New England and Canada. The  $A$ matrix is computed using the MatPower package (which is supported on Matlab) \cite{mat_power}. It is verified that the eigen values of the $A$ matrix lies inside the unit circle and the link outage results in output directions, which are observable.
We consider two set of possible contingencies $(37-25, 36-23, 33-19, 39 - 9)$ (The corresponding lines have been marked in green in Fig. \ref{fig_39_BUS_System}). For this set of contingencies, the sensitivity values both $F_{ij}$'s and $S_{ij}$'s are plotted in Fig. \ref{fig_39_bus_bar} (a).  It can be observed the vector $E$ and $S$ manifest similar pattern. This signifies the uncertainties are not interacting much. The relative order of the links in terms of the sensitivity is preserved. This is also reflected in the value of interaction factor $I$, which is $0.003$.
 \begin{figure}[h]
\begin{center}
\mbox{
\subfigure[]{\scalebox{.8}{\includegraphics[width=2.6 in, height=2.4 in]{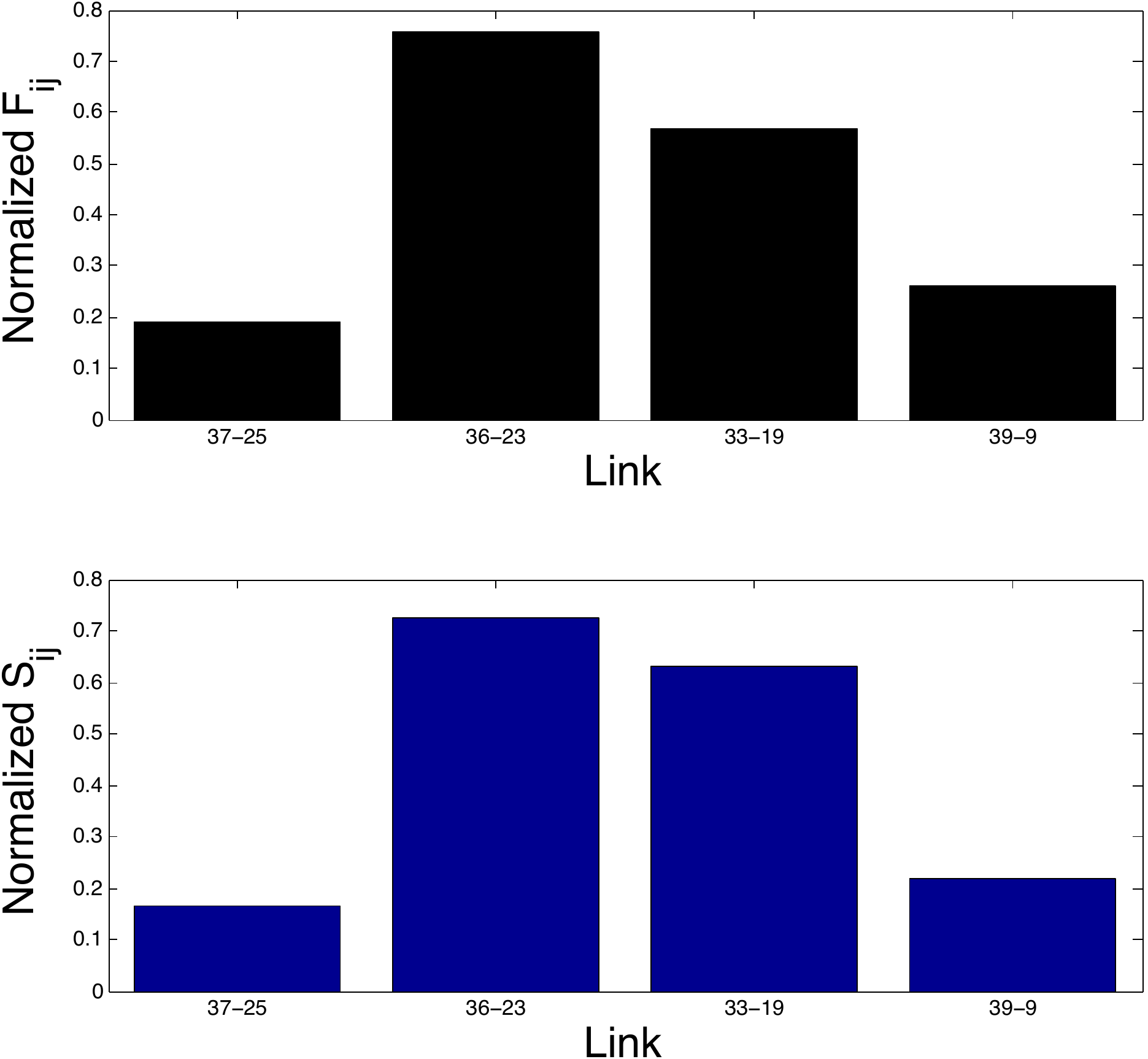}}}
\subfigure[]{\scalebox{.26}{\includegraphics[width=7.7 in, height=7.5 in]{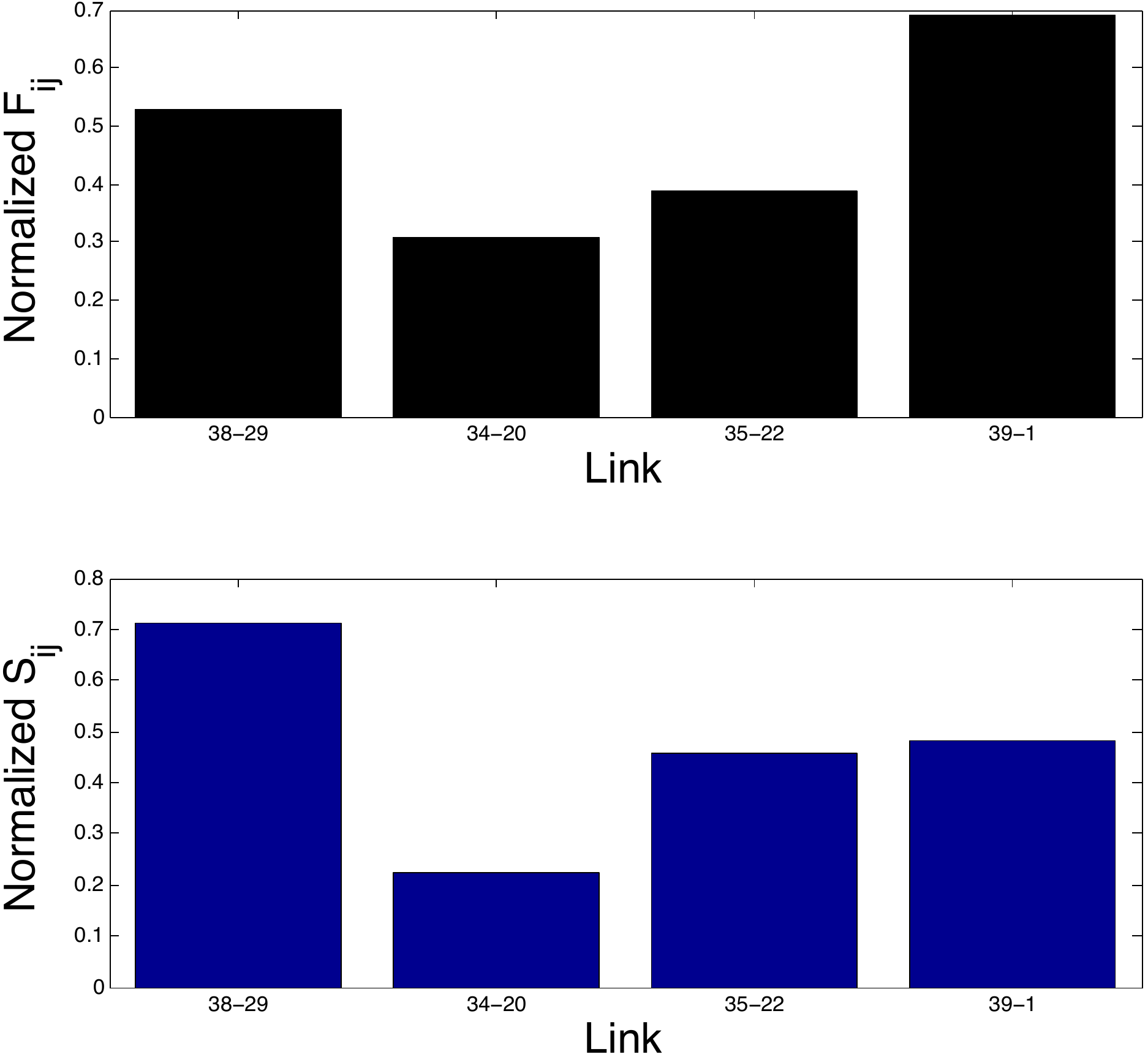}}} }
\caption{ Normalized sensitivity values for first set of contingencies - (a) marked in green in Fig. \ref{fig_39_BUS_System}, (b) marked in red in Fig. \ref{fig_39_BUS_System}. }
\label{fig_39_bus_bar}
\end{center}
\end{figure}
Next, we consider another set of contingencies $(38-29, 34-20, 35-22, 39-1)$  (The corresponding lines have been marked in red in Fig. \ref{fig_39_BUS_System}). For this set of contingencies, the two set of contingencies normalized $F_{ij}$'s and $S_{ij}$'s can be seen in Fig. \ref{fig_39_bus_bar} (b). It can be observed the vectors $E$ and $S$ show different trends, as the relative criticality changes in the two cases. This signifies the uncertainties are interacting and in effect the relative sensitivities of the links are changing. It can be seen that, the links that were most critical for single uncertainty becomes relatively more critical and vice versa. This is also reflected in the value of the interaction factor $I$, which becomes $0.10$ in this case.
\section{Conclusions} \label{con}
In this work, we propose a method to compute sensitivity of different uncertain links for $N-k$ contingency analysis in power network. We have proposed a gramian based method to compute the sensitivities. We have shown due to the interaction among uncertainties the ranking might get completely changed. This would motivate a rethinking of the contingency analysis, when multiple outages are considered.  This we have also witnessed in simulations for New England $39$ bus system. Also, we have proposed a method to compute interaction among a set of uncertainties.  In future work, we would pursue similar studies for transient stability.

\section{Appendix} \label{appen}
Here we would prove Theorem \ref{theorem_main1}. \\
\begin{proof}
By defining appropriate matrices the system in \eqref{system_robust-form} can be written in compact form as,
\begin{equation}
x(t+1) =\left (A+B \Delta(t) C \right) x(t) \label{system_robust-form1}
\end{equation}
where, $\Delta(t)$ comprises of the uncertainties $\delta_{k \ell} $'s.  The first step of proving Theorem \ref{theorem_main1} will be the following lemma. 
  \begin{lemma} \label{lem1} The system, described by \eqref{system_robust-form1},  is mean square exponentially stable only  if,   there exists a  symmetric, positive definite matrix $\hat{P}$  and $\epsilon > 0$, such that,
\begin{equation} \label{lyap1}
E_{{\Delta(t)}}\left [ \left (A + B \Delta(t) C  \right)^T {\hat P} \left(A + B \Delta(t) C  \right) \right]  - \hat{P} \le \epsilon I,
\end{equation}
and, $\gamma_0 \le \parallel \hat P \parallel \le \gamma_1  . $
\end{lemma}
\begin{proof}
Let us choose, $0<\beta< \beta_1<1$, where $\beta$ comes from the stability property described in Definition \ref{stab_def} . Next let us construct the matrix $\hat P $ as following, $\hat P : =  \sum_{n=0}^{\infty} \left ( \frac{1}{{{\beta}_1}^{^{n}}} \right ) E_{\Delta} \left [ \prod_{k=0}^n \left (A + B  \Delta (k) C  \right ) ^T  \prod_{k=0}^n \left (A + B \Delta(k) C \right )  \right]. $
From the construction and using the fact $\{ {\delta}_{k \ell} \}$ is a sequence of i.i.d. random variables, \[\hat{P} \ge  A^T A+ \sum_{(k, \ell) \in \mathcal{S}}\sigma_{k \ell}^2 \bar C^T_{k\ell}  \bar B^T_k \bar B_k C_{k\ell}. \]
According to Assumption \ref{ja_bound}, we get $   A^T A \ge L  > 0 $, which implies \eqref{lyap1}. Next, we prove the upper  bound on norm of $\hat P$.  Using mean square exponential stability of \eqref{system_robust-form1},
\begin{align*}
& \eta^T_0 \hat P \eta_0  <  K \sum_{n = 0}^{\infty}  {\left (\frac{{\beta}}{{\beta}_1}\right)^{n}} \eta_0^T ~ \eta_0 = \frac{K {\beta}_1}{\beta-{\beta}_1} \eta_0^T ~ \eta_0.
\end{align*}
We have already shown $\hat P \ge A^T  A \ge H>0$. Hence we can choose $\gamma_1, \gamma_2$ accordingly. From the construction we get,
\begin{align*}
& E_{{\Delta(t)}}\left [ \left (A + B \Delta(t) C  \right)^T { \hat P} \left(A + B \Delta(t) C  \right) \right]  < {\beta}_1 \hat{P}  \\
&= \hat{P} - (1 - {\beta}_1) \hat{P}\le \hat{P} \le  \hat{P} - \epsilon I, \text{where} ~ \epsilon > (1-{\beta}_1) {\gamma}_0 .
\end{align*}
\end{proof}
Now we are at a position to prove the Theorem \ref{theorem_main1}.
Equation (\ref{lyap1}) simplifies to,
\begin{align} \label{lyap_in}
 A^T \hat P A - \hat P  \le - \left ( \epsilon I + \sum_{(k, \ell) \in \mathcal{S}} \sigma_{k \ell}^2  C^T_{k\ell} \bar B^T_k \hat P \bar B_k C_{k\ell} \right ) .  
 \end{align}
 We define,
 \begin{align*}
 A^T \hat P A - \hat P  : = - \left (  \sum_{(k, \ell) \in \mathcal{S}} \sigma_{k \ell}^2  C^T_{k\ell} \bar B^T_k \hat P \bar B_k C_{k\ell} + R + \epsilon I \right ) .
 \end{align*}
for some , $R \ge 0$. In the next step, we decompose the positive definite matrix \[ \left ( \epsilon I + R +  \sum_{k, \ell} \sigma_{k \ell}^2  C^T_{k\ell}T \bar B^T_k \hat P  \bar B_k C_{k\ell}  \right ) .\] Let us define a set of basis vectors  $e^T_i \in \mathbb{R}^n$ as following, $ e_i  : = (\underbrace{0 \ldots, 0}_{i-1},1, \underbrace{0 \ldots, 0}_{N-i }). $
 With the aid of these vectors the identity matrix could be expressed as, $ I = \sum_{i} e^T_i e_i$.
$R$ is a positive semidefinite matrix. Let $r_j$'s be the left eigenvectors of $R$ corresponding to the positive eigenvalue ${\lambda}_j > 0$. This gives us the following decomposition $R = \sum_{j} \lambda_j  r^T r_j$. 
Now, $C_{k\ell}$'s  are not null row vectors from the observability assumption. Also let $ \mathcal{C}^l \le \parallel C_{k\ell} \parallel^2 \le \mathcal{C}^u$. We have denoted the set of all uncertain link as $\mathcal{S}=\{(k, \ell) | (k, \ell) ~ \text{is an uncertain link}\} $, and  we also denote cardinality of the set $\mathcal{S}$ by $N^*$ . We take an arbitrary ordered pair  $(k_1, {\ell}_1)$ from the set $\mathcal{S}$. Next, we decompose all other vectors $C_{p q}$'s, $e_i$'s, and $r_i $'s along this direction and its orthogonal direction. These decompositions could be expressed as following,
\begin{align}
C_{p q} &= \kappa^{pq}_{k _1 {\ell}_1}   C_{k_1 {\ell}_1} + \theta^{k_1 {\ell}_1}_{pq} , ~ C_{k_1 {\ell}_1} \left ( \theta^{k_1 {\ell}_1}_{pq} \right )^T = 0, \\
e_i &= \phi^{i}_{k_1 {\ell}_1}  C_{k_1 {\ell}_1} + \nu_{k_1 {\ell}_1}^{i} , ~ C_{k_1 {\ell}_1} \left ( \nu^{k_1 {\ell}_1}_{i} \right )^T = 0, \label{e_i_decom} \\
r_j &= \xi^{j}_{k_1 {\ell}_1}  C_{k_1 {\ell}_1} + \omega_{k_1 {\ell}_1}^{j}, ~ C_{k_1 {\ell}_1} \left ( \omega^{k_1 {\ell}_1}_{j} \right )^T = 0. 
\end{align}
where, $\kappa^{pq}_{k_1 \ell_1}  $, $\phi^{i}_{k_1 \ell_1}$, and $\xi^{i}_{k \ell} $ are scalars. 
 Let us now define the matrix, 
\begin{align*} 
 \mathcal{T}^0  &: = \sum_{p, q \in \mathcal{S}} \sigma_{p q}^2  C^T_{p q} \bar B^T_p \hat P \bar B_p C_{p q} +  \sum_{j}  \lambda_j   r^T_j  r^T_j + \frac{\epsilon}{N^*}  \sum_{i} e_i^T e_i.
\end{align*}
Using the decompositions of $C_{p q}$'s, $e_i$'s, and $r_i$'s we get,
\begin{align*}
& \mathcal{T}^0   =  \frac{\epsilon}{N^*}  \sum_{i} { \nu^{k_1 \ell_1}_{i}}^T  \nu^{k_1 \ell_1}_{i} +\sum_{p, q} \sigma_{p q}^2   \bar B^T_p \hat P \bar B_p {\theta^{k_1 \ell_1}_{pq}}^T \theta^{k_1 \ell_1}_{pq}  \\
& +   \sum_{j}  \lambda_j {\omega^{k_1 \ell_1}_{j}}^T \omega^{k_1 \ell_1}_{j} + (  \sum_{p, q} ( \kappa^{pq}_{k_1 \ell_1}  )^2\sigma_{p q}^2   \bar B^T_p \hat P \bar B_p \\ 
& +  \sum_{j}   (  \phi^{j}_{k_1 \ell_1}   )^2  \lambda_j    + \frac{\epsilon}{N^*}  \sum_{i}   ( \xi^{i}_{k_1 \ell_1}   )^2   ) C^T_{k_1 \ell_1} C_{k_1 \ell_1}, \\
& :=  \tilde{{\alpha}}^{k_1 \ell_1} C^T_{k_1 \ell_1} C_{k_1 \ell_1}  + \tilde {\mathcal{T}}_t^1.
\end{align*}
where,
\begin{align*}
& \tilde{{\alpha}}^{k_1 \ell_1}   = \sum_{p, q} ( \kappa^{pq}_{k_1 \ell_1}  )^2\sigma_{p q}^2   \bar B^T_p \hat P \bar B_p  +  \sum_{j}   (  \phi^{i}_{k_1 \ell_1}   )^2  \lambda_j+ \frac{\epsilon}{N^*}  \sum_{i}   ( \xi^{j}_{k_1 \ell_1}    )^2,  \\
 & \tilde {\mathcal{T}}^1  =  \sum_{p, q} \sigma_{p q}^2   \bar B^T_p \hat P \bar B_p {\theta^{k_1 \ell_1}_{pq}}^T \theta^{k_1 \ell_1}_{pq} +   \sum_{j}  \lambda_j  {\omega^{k_1 \ell_1}_{j}}^T \omega^{k_1 \ell_1}_{j}  \\
 & + \frac{\epsilon}{N^*}  \sum_{i} { \nu^{k_1 \ell_1}_{i}}^T  \nu^{k_1 \ell_1}_{i}.
\end{align*}
It can be noted that, $C_{k_1 \ell_1}  \tilde {\mathcal{T}}^1 = 0$. Next, we choose another element from $\mathcal{S}$, say $(k_2,{\ell}_2)$, corresponding to another uncertain link, and define, $ {\mathcal{T}}^1: =  \tilde {\mathcal{T}}^1 + \frac{\epsilon}{N^*}  \sum_{i} e_i^T e_i. $
And decompose the ${\mathcal{T}}^1$ as, $\mathcal{T}^1 =  \tilde{{\alpha}}^{k_2 {\ell}_2} {C_{k_2 {\ell}_2}}^T C_{k_2 {\ell}_2}  + \tilde {\mathcal{T}}^2. $
This process is continued until every uncertain link is exhausted. Then, we will have,
\begin{align}
&\sum_{(k, \ell) \in \mathcal{S}} \sigma_{k \ell}^2  C^T_{k\ell} \bar B^T_k \hat P \bar B_k C_{k\ell} + R + \epsilon I =  \sum_{i=0}^{N^* - 1} \mathcal{T}^i \nonumber \\
& = \sum_{(k, \ell) \in \mathcal{S}}  \tilde{{\alpha}}^{k \ell}  {C_{k \ell}}^T C_{k \ell} + \tilde {\mathcal{T}}^{N^*}. \label{1st_decom}
\end{align}
It can be noted that the matrix $\tilde {\mathcal{T}}^{N^*}$ is positive semi definite. Let $t_j$'s be the left  eigenvectors of   $\tilde {\mathcal{T}}^{N^*}$ corresponding to the each of the eigenvalue ${\Lambda}_j  > 0$. This gives, $ \tilde {\mathcal{T}}^{N^*} = \sum_j {\Lambda}_j  {t^T _j} t_j. $
 Again, we do similar decomposition for the matrix $\tilde {\mathcal{T}}^{N^*} $ .We decompose the vectors $t_j $ along the directions $C_{k_1 \ell_1}$ for $(k_1, \ell_1) \in \mathcal{S}$,
\begin{align}
 t_j = \Xi^{j}_{k_1 \ell_1}  C_{k_1 \ell_1} + \Omega_{k_1 \ell_1}^{j}, ~ C_{k_1 \ell_1} \left ( \Omega^{k_1 \ell_1}_{j} \right )^T = 0.  \label{t_decom} 
 \end{align}
Next, we define, $ \tilde {\mathcal{U}}^{0} := \tilde {\mathcal{T}}^{N^* -1} = \sum_j {\Lambda}_j  {t^T _j}  t_j. $ Using Eq. \eqref{t_decom},
\begin{align*}
& \tilde {\mathcal{U}}^{0}   = \left ( \Xi^{j}_{k_1 \ell_1} \right) ^2  C^T _{k_1 \ell_1} C_{k_1 \ell_1} + { \Omega_{k_1 \ell_1}^{j}}^T  \Omega_{k_1 \ell_1}^{j}  \\
& : = \hat{{\alpha}}^{k_1 \ell_1} C^T _{k_1 \ell_1} C_{k_1 \ell_1} + \tilde {\mathcal{U}}^{1}.
\end{align*}
where, $\hat{{\alpha}}^{k_1 \ell_1}  = \left ( \Xi^{j}_{k_1 \ell_1} \right) ^2 \ge 0$. Next, we choose another element $(k_2, {\ell}_2) \in \mathcal{S}$ and decompose $\tilde {\mathcal{U}}^{1}$. This gives us, $ \tilde {\mathcal{U}}^{1}= \hat{{\alpha}}^{k_2 {\ell}_2} C^T _{k_2 {\ell}_2} C_{k_2 {\ell}_2} + \tilde {\mathcal{U}}^{2}. $ This procedure is continued until the set $\mathcal{S}$ is exhausted. In this case, we define ${Q} : = \tilde {\mathcal{U}}^{N^*}$. This gives us,
\begin{align}
 \tilde {\mathcal{T}}^{N^*} = \tilde {\mathcal{U}}^{0} = \sum_{(k \ell) \in \mathcal{S} } \hat{{\alpha}}^{k \ell}  C^T _{k \ell} C_{k \ell} + {Q}. \label{2nd_decom}
\end{align}
where, $\hat{{\alpha}}^{k \ell}  \ge 0$ and $C_{k \ell} {Q} = 0, \forall k, \ell$.
Combining Eq. \eqref{1st_decom} and \eqref{2nd_decom} we get,
\[ \sum_{(k, \ell) \in \mathcal{S}} \sigma_{k \ell}^2  C^T _{k\ell} \bar B^T_k \hat P B_k C_{k\ell} + R+ \epsilon I =  \sum_{(k \ell) \in \mathcal{S} } {\alpha}^{k \ell} C^T _{k \ell} C_{k \ell} + Q, \] where, ${\alpha}^{k \ell} : = \tilde{{\alpha}}^{k \ell} +\hat{{\alpha}}^{k \ell}$. It is to be noted that ${\alpha}^{k \ell} < \frac{{\gamma}_1}{\mathcal{C}^l}$ (where the $\gamma_1$ comes from the upper bound of the matrix $\hat P$ from Lemma \ref{lem1}, and ${\mathcal{C}^l}$ is the lower bound of $\parallel C^T _{k\ell} \parallel , \forall (k, \ell) \in \mathcal{S}$).  Next, we would show a lower bound on ${\alpha}^{k \ell}$ for all $(k, \ell) \in \mathcal{S}$. For that purpose, let us consider the Eq. \eqref{e_i_decom}, $ e_i = \phi^{i}_{k \ell}  C_{k \ell} + \nu_{k \ell}^{i}, ~ C_{k \ell} \left ( \nu^{k \ell}_{i} \right )^T = 0. $
Multiplying both sides by $C^T _{k \ell}$ from right, $ \sum_i |\phi^{i}_{k \ell} |^2  = \frac{1}{\parallel C_{k \ell}  \parallel^2} \ge \frac{1}{\mathcal{C}^u} $ .
This gives us,
\begin{align*}
\tilde{{\alpha}}^{k \ell}   & = \sum_{p, q} ( \kappa^{pq}_{k \ell}  )^2\sigma_{p q}^2   \bar B^T_p \hat P B_p  +  \sum_{j}   (  \phi^{i}_{k \ell}   )^2  \lambda^i  \\
& + \frac{\epsilon}{N^*}  \sum_{i}   ( \xi^{j}_{k \ell}    )^2  \ge \frac{\epsilon}{\mathcal{C}^u N^*} \ge \frac{\epsilon}{\mathcal{C}^u M^2 n^2}, ~ ~~\text{as}~ \left ( M n\right )^2 \ge N^*.
\end{align*}
This means ${\alpha}^{k \ell}  = \tilde{{\alpha}}^{k \ell} +\hat{{\alpha}}^{k \ell}  \ge \frac{\epsilon}{\mathcal{C}^u M^2 n^2}, ~ \forall (k, \ell) \in \mathcal{S} $. Hence,
\begin{align*}
\sum_{(k, \ell) \in \mathcal{S}} \sigma_{k \ell}^2  C^T _{k\ell} \bar B^T_k \hat P \bar B_k C_{k\ell} + R + \epsilon I  =\sum_{k \ell } {\alpha}^{k \ell}  C^T _{k \ell} C_{k \ell} + Q,
\end{align*} 
Finally, we can say,
\begin{align}
 A^T  \hat P A - \hat P  
 = -\left (  \sum_{(k, \ell) \in \mathcal{S}} {\alpha}^{k \ell} C^T_{k \ell} C_{k \ell} + {Q} \right ),  \label{final_decom}
\end{align} 
$\frac{{\gamma}_1}{\mathcal{C}^l} > {\alpha}^{k \ell} \ge \frac{\epsilon}{v_{k \ell}^2 M^2 n^2}$ and $C_{k \ell}  {Q} = 0, \forall (k, \ell) \in \mathcal{S}$.  At this stage, we use the upper and lower bound property of $\alpha^{k \ell}$ to reach the condition described in Theorem \ref{theorem_main1}. Eq. \eqref{lyap_in} further necessarily implies, $ A^T \hat P A - \hat P +\sum_{(k, \ell) \in \mathcal{S}}\sigma_{k \ell}^2  C^T_{k\ell} \bar B^T_k \hat P \bar B_k C_{k\ell} \le  0. $ Combining with Eq. \eqref{final_decom}, we get $ \sum_{(k, \ell) \in \mathcal{S}} \left ( \sigma_{k \ell}^2   \bar B^T_k \hat P \bar B_k  - {\alpha}^{k \ell}\right) C^T_{k\ell} C_{k\ell} - {Q} \le  0.$ Using the fact $C_{k \ell}  {Q} = 0, ~ \forall k, \ell$, we obtain  $ \sum_{(k, \ell) \in \mathcal{S}} \left ( \sigma_{k \ell}^2   \bar B^T_k \hat P \bar B_k  - {\alpha}^{k \ell} \right) C^T_{k\ell} C_{k\ell}  \le  0.$
 Taking trace on both sides,
 \begin{equation}
 \sum_{(k, \ell) \in \mathcal{S}}  \left ( \sigma_{k \ell}^2  \bar B^T_k \hat P \bar B_k - {\alpha} ^{k \ell} \right)  C_{k\ell}   C^T_{k\ell} \le 0. \label{trace_cnd}
 \end{equation}
From Eq. (\ref{final_decom}),
\begin{align*} 
& \hat{P}= \sum_{i=0}^{\infty}  \left [ \left ( \prod_{m=0}^i A \right ) ^T  \left ( \left (\sum_{(k, \ell) \in \mathcal{S}} {\alpha}^{k \ell} C^T_{k\ell}  C_{k\ell} \right ) - Q \right) \left ( \prod_{m=0}^i A \right ) \right] .
 \end{align*}
Since $Q\leq 0$, it follows that $\tilde P \leq \hat P$, where $\tilde P$ is defined as follows: \[ \tilde{P} := \sum_{i=0}^{\infty} \left ( \prod_{m=0}^i A \right ) ^T \left (\sum_{(k, \ell) \in \mathcal{S}} {\alpha}^{k \ell} C^T_{k\ell}   C_{k\ell} \right )  \left ( \prod_{m=0}^i A \right). \]
 Using the above definition it follows that $\tilde P$ satisfies, $ A^T \tilde P A  -\tilde{P} = -  \left (\sum_{(k, \ell) \in \mathcal{S}} {\alpha}^{k \ell}  C^T_{k\ell}  C_{k\ell} \right ). $ 
 Substituting $\tilde P$ for $\hat P $ in (\ref{trace_cnd}), we get,  $  \sum_{(k, \ell) \in \mathcal{S}}  \left ( \sigma_{k \ell}^2  \bar B^T_k  {\tilde P} \bar B_k - {\alpha}^{k \ell}  \right)  C_{k\ell}   C^T_{k\ell} \le 0. $ 
We also define, 
 $ {\alpha}_* = \min_{k\ell} \alpha^{k \ell}$. 
We construct the matrix $P$ as, \[ P : = \sum_{i=0}^{\infty} \left ( \prod_{m=0}^i A \right ) ^T \left (\sum_{(k, \ell) \in \mathcal{S}}  C^T_{k\ell}   C_{k\ell} \right )  \left ( \prod_{m=0}^i A \right)  \] . The existence of matrices ${P} = P^T > 0$ is guaranteed by observability of the pair $(A,C_{k \ell})$. It can be noted,
$ A^T {P} A - P  =  -  \sum_{(k, \ell) \in \mathcal{S}} C^T_{k \ell}   C_{k \ell} . $  Also, it can be observed that ${\alpha}_* P \le \tilde{P} \le \hat{P}$. This means $  \sum_{(k, \ell) \in \mathcal{S}}  \left ( \sigma_{k \ell}^2  \bar B^T_k  \alpha_* P \bar B_k - {\alpha} ^{k \ell}  \right)  C_{k\ell}  C^T_{k\ell}  \le 0 , $  and,
From \eqref{final_decom},
\begin{align}
  &\sum_{(k, \ell) \in \mathcal{S}}  \sigma_{k \ell}^2     \bar B^T_k  P  \bar B_k   C_{k\ell}   C^T_{k\ell}  \le   \sum_{(k, \ell) \in \mathcal{S}}\frac{ {\alpha}^{k \ell}}{ {\alpha}_*}  C_{k\ell}   C^T_{k\ell}, \nonumber \\
   &\sum_{(k, \ell) \in \mathcal{S}}  \sigma_{k \ell}^2     \bar B^T_k  P  \bar B_k   C_{k\ell}   C^T_{k\ell} \le  \bar \alpha \sum_{(k, \ell) \in \mathcal{S}}  C_{k\ell}   C^T_{k\ell}.  \nonumber
  \end{align}
  We use the fact $\frac{{\gamma}_1}{\mathcal{C}^l} > {\alpha}^{k \ell} \ge \frac{\epsilon}{\mathcal{C}^u M^2 n^2}$ to get $\frac{ {\alpha}^{k \ell} }{ {\alpha}_* } < \frac{{\gamma}_1 \mathcal{C}^u M^2 n^2}{\epsilon \mathcal{C}^l} : = \bar \alpha. $  It can be noted that $\bar{\alpha} \ge 1$. Hence the proof.
  \end{proof}
  
\bibliographystyle{plain}
\bibliography{ref_new}  

\begin{thebibliography}{1}

\bibitem{Boyd_book}
S.~Boyd, L.~El Ghaoui, E.~Feron, and V.~Balakrishnan.
\newblock {\em Linear Matrix Inequalities in System and Control Theory}.
\newblock SIAM, 1994.

\bibitem{chen2005identifying}
Qiming Chen and James~D McCalley.
\newblock Identifying high risk nk contingencies for online security
  assessment.
\newblock {\em IEEE Transactions on Power Systems}, 20(2):823--834, 2005.

\bibitem{paganini_dull}
G.~E. Dullerud and F.~Paganini.
\newblock {\em A Course in Robust Control Theory}.
\newblock Springer-Verlag, New York, 1999.

\bibitem{cont_transience_1}
H.K. Nam, K.S. Shim, Y.K. Kim, S.G. Song, and K.Y. Lee.
\newblock Contingency ranking for transient stability via eigen-sensitivity
  analysis of small signal stability model.
\newblock In {\em Power Engineering Society Winter Meeting, 2000. IEEE},
  volume~2, pages 861--865 vol.2, 2000.

\bibitem{pasqualetti2011cyber}
Fabio Pasqualetti, Florian Dorfler, and Francesco Bullo.
\newblock Cyber-physical attacks in power networks: Models, fundamental
  limitations and monitor design.
\newblock In {\em 2011 50th IEEE Conference on Decision and Control and
  European Control Conference}, pages 2195--2201, 2011.

\bibitem{cascade_1}
M.~Vaiman, K.~Bell, Y.~Chen, B.~Chowdhury, I.~Dobson, P.~Hines, M.~Papic,
  S.~Miller, and P.~Zhang.
\newblock Risk assessment of cascading outages: Methodologies and challenges.
\newblock {\em IEEE Transactions on Power Systems}, 27(2):631--641, 2012.

\bibitem{wood2012power}
Allen~J Wood and Bruce~F Wollenberg.
\newblock {\em Power generation, operation, and control}.
\newblock John Wiley \& Sons, 2012.

\bibitem{wu2011probability}
Xu~Wu, Jianhua Zhang, and Qian Chen.
\newblock Probability analysis model and risk assessment of nk contingency
  based on condition-based maintenance.
\newblock In {\em 2011 4th International Conference on Electric Utility
  Deregulation and Restructuring and Power Technologies (DRPT)}, pages
  968--973. IEEE, 2011.

\bibitem{mat_power}
R.D. Zimmerman, C.E. Murillo-S‡nchez, and R.J.Thomas.
\newblock Matpower: Steady-state operations, planning, and analysis tools for
  power systems research and education.
\newblock {\em IEEE Transactions on Power Systems}, 26(1):12--19, 2011.
\newblock {M}atpower Package - http://www.pserc.cornell.edu//matpower/.

\end{thebibliography}
  
\end{document}